\documentclass[dvipdfmx]{amsart}

\makeatletter
\@addtoreset{equation}{section}

\makeatother

\usepackage{amsfonts, amsmath, amssymb, amsthm}
\usepackage{url}
\usepackage{cleveref}
\usepackage[dvips]{graphicx}
\usepackage[dvips]{color}
\usepackage{amscd}
\usepackage[right]{lineno}
\usepackage{MnSymbol}
\usepackage{enumerate}

\usepackage[all]{xy}

\newtheorem{theorem}{Theorem}[section]
\newtheorem{lemma}{Lemma}[section]
\newtheorem{proposition}{Proposition}[section]
\newtheorem{corollary}{Corollary}[section]
\theoremstyle{remark}

\newtheorem{remark}{Remark}[section]

\newcommand{\teich}{\mathcal{T}}

\newcommand{\ext}{{\rm Ext}}

\newcommand{\Bers}[1]{\mathcal{T}^B_{#1}}

\newcommand{\ThursM}{\mu_{Th}}
\newcommand{\PThursM}{\hat{\mu}_{Th}}

\newcommand{\convgenus}{{\boldsymbol \xi}}
\newcommand{\proje}{\boldsymbol{\psi}}

\newcommand{\pushThursMBers}{{\boldsymbol \mu}^B}

\newcommand{\ray}{{\boldsymbol r}}

\makeindex

\begin{document}
%\frontmatter

\title[Poisson integral formula]{Bounded {P}luriharmonic {F}unctions and {H}olomorphic {F}unctions on {T}eichm\"uller {S}pace II \\
-- Poisson integral formula --}
\author{Hideki Miyachi}

\date{\today}
\dedicatory{Dedicated to Professor~Yunping Jiang on the~occasion of his~65-th birthday.}
\address{School of Mathematics and Physics,
College of Science and Engineering,
Kanazawa University,
Kakuma-machi, Kanazawa,
Ishikawa, 920-1192, Japan
}
\email{miyachi@se.kanazawa-u.ac.jp}
\thanks{This work is partially supported by JSPS KAKENHI Grant Numbers
25K00909,
23K22396 ,
20K20519}

\begin{abstract}
In this paper, we establish the Poisson integral formula for bounded pluriharmonic functions on the Teichm\"uller space of analytically finite Riemann surfaces of type $(g,m)$ with $2g-2+m>0$. We also discuss a version of the F. and M. Riesz theorem concerning the value distribution of plurisubharmonic functions on the Teichm\"uller space, as well as a Teichm\"uller-theoretic interpretation of the mean value theorem for pluriharmonic functions.
\end{abstract}

\maketitle
\tableofcontents

\section{Introduction}

This paper is a continuation of the author's previous works \cite{MR4028456}, \cite{MR4633651}, and \cite{MR4830064}.  
A part of the results are announced in \cite{bounded-Function_Ergodic}.

\subsection{Background}

Fix a base point $x_0 \in \teich_{g,m}$, and let $\Bers{x_0}$ denote the Bers slice with base point $x_0$. Let $\partial \Bers{x_0}$ be the Bers boundary of $\Bers{x_0} \cong \teich_{g,m}$, and let $\partial^{\mathrm{ue}} \Bers{x_0}$ denote the subset of $\partial \Bers{x_0}$ consisting of totally degenerate groups without accidental parabolic transformations, whose ending laminations support minimal, filling, and uniquely ergodic measured laminations (see \S\ref{subsec:ml} and \S\ref{subsec:BoundarygroupswithoutAPT} for definitions).

We define a function $\mathbb{P} \colon \teich_{g,m} \times \teich_{g,m} \times \partial \Bers{x_0} \to \mathbb{R}$ by
\begin{equation}
\label{eq:Poisson-kernel}
\mathbb{P}(x,y,\varphi) =
\begin{cases}
{\displaystyle
\left(
\frac{\ext_x(\lambda_\varphi)}{\ext_y(\lambda_\varphi)}
\right)^{3g - 3 + m}
}
& \text{if } \varphi \in \partial^{\mathrm{ue}} \Bers{x_0}, \\
1 & \text{otherwise},
\end{cases}
\end{equation}
where $\lambda_\varphi$ is the measured lamination whose support is the ending lamination of $\varphi$.

Let $\pushThursMBers_{x} = {\boldsymbol \mu}^{B,x_0}_x$ be the pushforward of the normalized Thurston measure associated to $x \in \teich_{g,m}$ via the canonical projection to $\mathcal{PML}$ (see \S\ref{subsec:pushforwardmeasure-Bers-boundary}). The measure $\pushThursMBers_x$ on $\partial \Bers{x_0}$ is the pluriharmonic measure on $\teich_{g,m}$ with pole at $x$ in the sense of Demailly \cite{MR881709} (cf.\ \cite{MR4633651}).

For an integrable function $g$ on $\partial \Bers{x_0}$ with respect to the measure $\pushThursMBers_{x_0}$, we define the \emph{Poisson integral} by
\begin{equation}
\label{eq:PI}
\mathbb{P}(g)(x) = \int_{\partial \Bers{x_0}} g(\varphi) \, d\pushThursMBers_x(\varphi)
= \int_{\partial \Bers{x_0}} g(\varphi) \, \mathbb{P}(x_0,x,\varphi) \, d\pushThursMBers_{x_0}(\varphi)
\end{equation}
for $x \in \teich_{g,m}$.

\subsection{Main theorem}
The main theorem generalizes the result of \cite{MR4633651}, which establishes a full analogue of Fatou's theorem for the unit disk (\cite{MR1555035}).

\begin{theorem}[Poisson integral formula]
\label{thm:main1}
Fix $x_0 \in \teich_{g,m}$.
Let $u$ be a bounded pluriharmonic function on $\teich_{g,m}$, and let $u^*$ denote its radial limit.
Then,
\[
u(x) = \mathbb{P}(u^*)(x)
\]
for all $x \in \teich_{g,m}$.
\end{theorem}

The notion of radial limits for bounded pluriharmonic functions on $\teich_{g,m}$, as it appears in \Cref{thm:main1}, was formulated in \cite{MR4830064}. We briefly recall this notion in \S\ref{sec:radial_limits}.

As a consequence of \Cref{thm:main1}, we obtain the following.

\begin{corollary}
\label{coro:main1dash}
Let $\mathcal{PH}^\infty(\teich_{g,m})$ denote the Banach space of bounded pluriharmonic functions on $\teich_{g,m}$, equipped with the supremum norm.
The radial limit map
\[
\mathcal{PH}^\infty(\teich_{g,m}) \ni u \mapsto u^* \in L^\infty(\partial\Bers{x_0})
\]
is a linear isometric injection. Moreover, its left inverse is given by the Poisson integral:
\[
L^\infty(\partial\Bers{x_0}) \ni g \mapsto \mathbb{P}(g) \in C^1_b(\teich_{g,m}),
\]
where $C^1_b(\teich_{g,m})$ denotes the space of bounded $C^1$ functions on $\teich_{g,m}$.

In particular, a bounded pluriharmonic function on $\teich_{g,m}$ is constant if and only if its radial limit is constant almost everywhere.
\end{corollary}

%The ``only if" part of \Cref{coro:main1dash} is trivial. The ``if" part follows from \Cref{coro:main1dash}.

\subsection{Value distributions of plurisubharmonic functions}
By applying the proof of \Cref{thm:main1}, we study the distribution of values of non-negative bounded log-plurisubharmonic functions on $\teich_{g,m}$ via their boundary behavior. In particular, we establish an inequality of F. and M. Riesz type:

\begin{theorem}[F. and M. Riesz-type inequality]
\label{thm:main2}
Let $v$ be a non-negative plurisubharmonic function on $\teich_{g,m}$ such that $v \le M$ for some $M > 0$, and assume that $\log v$ is also plurisubharmonic.

Suppose there exists a measurable decomposition $\{E_k\}_{k\in I}$ of $\partial\Bers{x_0}$ and constants $0 < m_k \le M$ such that
\[
\limsup_{t \to \infty} v\left(\ray_\lambda^x(t)\right) \le m_k
\]
for all $k \in \mathbb{N}$ and $\lambda \in \mathcal{ML}$ with $\varphi_\lambda \in E_k$, where $\ray_\lambda^x \colon [0, \infty) \to \teich_{g,m}$ denotes the Teichmüller geodesic ray from $x \in \teich_{g,m}$ defined by the Hubbard-Masur differential for $\lambda$.
Then,
\[
v(x) \le 
\prod_{k \in I} m_k^{\pushThursMBers_x(E_k)}.
\]
\end{theorem}

A collection $\{E_k\}_{k \in I}$ of disjoint measurable subsets of a measure space $(X,\mu)$, indexed by a countable set $I$, is called a \emph{measurable decomposition} of $X$ if $X \setminus \bigcup_{k \in I} E_k$ is a null set.
As a consequence of \Cref{thm:main2}, we obtain the following result, which was also observed in \cite{MR4830064}:

\begin{corollary}[Identity theorem]
\label{coro:identity_theorem}
Let $f$ be a bounded holomorphic function on $\teich_{g,m}$. If its radial limit $f^*$ vanishes on a measurable subset of $\partial\Bers{x_0}$ with positive measure, then $f$ vanishes identically on $\teich_{g,m}$.
\end{corollary}

\subsection{About the paper}
In \S\ref{sec:notation} and \S\ref{sec:Kleinian-surface-groups-and-the-Bers-slice}, we introduce the notations and conventions used throughout this paper.  
In \S\ref{sec:radial_limits}, we review the radial limit theorem and the disintegration method developed in \cite{MR4830064}.  
In \S\ref{sec:Proofs_theorems}, we prove \Cref{thm:main1}, \Cref{coro:main1dash}, and \Cref{thm:main2}.  
In \S\ref{sec:MVT-JM}, we explore the mean value theorem from the Teichmüller-theoretic point of view and interpret the Thurston measure as a Jensen measure.  
In the final section, \S\ref{sec:conclusion}, we discuss related open problems and directions for future research.
.

\subsection*{Acknowledgement}
The author would like to express his sincere gratitude to Professors Sudeb Mitra and Wang Zhe for the opportunity to participate in the wonderful conference in celebration of Professor Yunping Jiang's birthday and to present the author's research at the conference.
He also thanks Professor Ara Basmajian for fruitful discussions on applications of the Poisson integral formula. Additionally, he is grateful to Professor Toshiyuki Sugawa for valuable discussions on the value distribution of subharmonic functions.

The author further wishes to thank Professors Athanase Papadopoulos and Ken'ichi Ohshika for their warm encouragement and continuous support throughout his research.

\section{Notation}
\label{sec:notation}
%\section{Teichm\"uller theory}
%\label{sec:Teichmuller-theory}
Let $\Sigma_{g,m}$ denote a closed, orientable surface of genus $g$ with $m$ marked points, where $2g - 2 + m > 0$ (allowing $m = 0$).  
The (topological) complexity of $\Sigma_{g,m}$ is defined as  
$\convgenus = \convgenus(\Sigma_{g,m}) = 3g - 3 + m$.  
In this section, we review fundamental aspects of Teichmüller theory.  
For further details, the reader is referred to \cite{MR590044}, \cite{MR568308}, \cite{MR903027}, \cite{MR2245223}, \cite{MR1215481}, and \cite{MR927291}.

\subsection{Teichm\"uller space}
A \emph{marked Riemann surface} of type $(g,m)$  
is a pair $(M,f)$ consisting of a Riemann surface $M$ of analytically finite type $(g,m)$  
together with an orientation-preserving homeomorphism  
$f \colon \Sigma_{g,m} \to M$.  
The \emph{Teichm\"uller space} $\teich_{g,m}$ of Riemann surfaces of type $(g,m)$  
is defined as the set of Teichm\"uller equivalence classes of marked Riemann surfaces of type $(g,m)$, where two marked Riemann surfaces $(M_1, f_1)$ and $(M_2, f_2)$  
are said to be \emph{Teichm\"uller equivalent}  
if there exists a conformal mapping $h \colon M_1 \to M_2$  
such that $h \circ f_1$ is homotopic to $f_2$.

The \emph{Teichm\"uller distance} $d_T$ is a metric on $\teich_{g,m}$ defined by
$$
d_T(x_1, x_2) = \frac{1}{2} \log \inf_h K(h),
$$
for $x_i = (M_i, f_i)$ $(i=1,2)$,  
where the infimum is taken over all quasiconformal mappings  
$h \colon M_1 \to M_2$ homotopic to $f_2 \circ f_1^{-1}$,  
and where $K(h)$ denotes the maximal dilatation of $h$.  
It is known that $d_T$ is complete.

%The \emph{mapping class group} $\mcg_{g,m}$  
%is the group of homotopy classes of orientation-preserving homeomorphisms of $\Sigma_{g,m}$.  
%Each element $[\omega] \in \mcg_{g,m}$ acts on $\teich_{g,m}$ by
%$$
%[\omega](M,f) = (M, f \circ \omega^{-1}).
%$$
\subsection{Infinitesimal complex structure on $\teich_{g,m}$}
\label{subsec:infinitesimal-Teich}
For $x = (M,f) \in \teich_{g,m}$,  
we denote by $\mathcal{Q}_{x}$ the complex Banach space  
of holomorphic quadratic differentials $q = q(z)dz^{2}$ on $M$ satisfying
$$
\|q\| = \iint_{M} |q(z)| \frac{\sqrt{-1}}{2} dz \wedge d\overline{z} < \infty.
$$
By virtue of the Riemann-Roch theorem, we have $\dim_{\mathbb{C}} \mathcal{Q}_{x} = \convgenus$.  
The union $\mathcal{Q}_{g,m} = \bigcup_{x \in \teich_{g,m}} \mathcal{Q}_{x}$  
is identified with the holomorphic cotangent bundle of $\teich_{g,m}$  
via the pairing \eqref{eq:pairing-T-coT} defined later.

% A differential $q \in \mathcal{Q}_{g,m}$ is said to be \emph{generic}  
% if all zeros are simple and all marked points of the underlying surface  
% are simple poles of $q$.
% The set of generic differentials forms an open and dense subset  
% of $\mathcal{Q}_{g,m}$ and of each fiber $\mathcal{Q}_{x}$ for $x \in \teich_{g,m}$.

The Teichm\"uller space $\teich_{g,m}$ is a complex manifold of dimension
$\convgenus$.  
Its infinitesimal complex structure is described as follows.

Let $x=(M,f)\in \teich_{g,m}$.  
Let $L^{\infty}(M)$ be the Banach space of measurable $(-1,1)$-forms
$\mu=\mu(z) d\overline{z}/dz$ on $M$ equipped with the essential supremum norm
$$
\|\mu\|_{\infty} = {\rm ess.sup}_{p \in M} |\mu(p)| < \infty.
$$
Then the holomorphic tangent space $T_{x}(\teich_{g,m})$ of $\teich_{g,m}$ at $x$
is described as the quotient space
$$
L^{\infty}(M) / \left\{ \mu \in L^{\infty}(M) \mid
\llangle \mu, \varphi \rrangle = 0, \ \forall \varphi \in \mathcal{Q}_{x} \right\},
$$
where
\begin{equation*}
%\label{eq:pairing}
\llangle \mu, \varphi \rrangle =
\iint_{M} \mu \varphi = \iint_{M} \mu(z) \varphi(z) \frac{\sqrt{-1}}{2} dz \wedge d\overline{z}.
\end{equation*}
Any element of $L^\infty(M)$ is called an \emph{infinitesimal Beltrami differential} in this context.  
For $v = [\mu] \in T_x(\teich_{g,m})$
and $\varphi \in \mathcal{Q}_x$,
a \emph{canonical pairing} between $T_x(\teich_{g,m})$
and $\mathcal{Q}_x \cong T_x^*(\teich_{g,m})$ is defined by
\begin{equation}
\label{eq:pairing-T-coT}
\langle v, \varphi \rangle = \llangle \mu, \varphi \rrangle.
\end{equation}
%From the definition,
%the pairing \eqref{eq:pairing} defines a pairing between $T_{x}\teich_{g,m}$
%and $\mathcal{Q}_{x}$.
%From this observation,
%the space $\mathcal{Q}_{x}$ is identified with the holomorphic cotangent
%space at $x$ of $\teich_{g,m}$.
%
%Let $q_0\in \mathcal{Q}_x$ be a generic differential and $v\in T_x\teich_{g,m}$.
%The \emph{$q_0$-realization of $v$}
%is a unique holomorphic quadratic differential
%$\eta_v\in \mathcal{Q}_x$ which satisfies
%\begin{equation}
%\label{eq:q0-realization}
%\langle v,\varphi\rangle=\int_M\frac{\overline{\eta_v}}{|q_0|}\varphi
%\end{equation}
%for all $\varphi\in \mathcal{Q}_x$.
%The correspondence
%\begin{equation}
%\label{eq:qo-realization-1}
%T_x\teich_{g,m}\ni v\mapsto \eta_v\in \mathcal{Q}_x
%\end{equation}
%is a complex anti-linear isomorphism
%(cf. \cite[Theorem 5.3]{MR3413977} and \cite[\S4.2]{MR3715450}. See also \cite[p.128]{MR0996636}).
%The Hermitian inner product appeared in the right-hand side \eqref{eq:q0-realization} is initiated by Dumas in \cite[\S5]{MR3413977}.
% for arbitrary holomorphic quadratic differentials.

\subsection{Measured laminations}
\label{subsec:ml}
Let $\mathcal{S}$ denote the set of homotopy classes of essential simple closed curves on $\Sigma_{g,m}$.  
Let $i(\alpha, \beta)$ denote the \emph{geometric intersection number} of simple closed curves $\alpha, \beta \in \mathcal{S}$.  
%A \emph{multi-curve} is defined as an unordered finite sequence $(\alpha_i)_i$ in $\mathcal{S}$ such that $\alpha_i \neq \alpha_j$ and $i(\alpha_i, \alpha_j) = 0$ for all $i \neq j$.

Fix a complete hyperbolic structure of finite area on $\Sigma_{g,m}$.  
A \emph{geodesic lamination} $L$ on $\Sigma_{g,m}$ is a non-empty closed subset that is a disjoint union of complete simple geodesics, where a geodesic is said to be \emph{complete} if it is either closed or has infinite length in both directions.  
The geodesics in $L$ are referred to as the \emph{leaves} of $L$.

A \emph{transverse measure} for a geodesic lamination $L$ assigns a Borel measure to each arc transverse to $L$, satisfying the following two conditions:  
(1) if an arc $k'$ is contained in a transverse arc $k$, then the measure assigned to $k'$ is the restriction of the measure assigned to $k$;  
(2) if two arcs $k$ and $k'$ are homotopic through a family of arcs transverse to $L$, then the homotopy transports the measure on $k$ to that on $k'$.  
A transverse measure on a geodesic lamination $L$ is said to have \emph{full support} if, for each transverse arc $k$, the support of the assigned measure is precisely $k \cap L$.

A \emph{measured lamination} $\lambda$ is a pair consisting of a geodesic lamination $|\lambda|$, called the \emph{support} of $\lambda$, together with a transverse measure on $|\lambda|$ of full support.  
Let $\mathcal{ML}$ denote the \emph{space of measured laminations} on $\Sigma_{g,m}$ with respect to the fixed hyperbolic structure.  
A weighted simple closed curve $t\alpha$ is naturally identified with a measured lamination whose support is a simple closed geodesic homotopic to $\alpha$, and whose transverse measure is given by $t$ times the Dirac measure supported at the intersections with transverse arcs.  
The positive real numbers $\mathbb{R}_+$ act on $\mathcal{ML}$ by multiplication. The quotient space $\mathcal{PML} = (\mathcal{ML} \setminus \{0\})/\mathbb{R}_+$ is called the \emph{space of projective measured laminations}.

A sequence $\{\lambda_n\}_n \subset \mathcal{ML}$ is said to converge to a measured lamination $\lambda \in \mathcal{ML}$ if, for any generic arc $k$, the transverse measures on $k$ assigned by $\lambda_n$ converge weakly to the transverse measure assigned by $\lambda$.  
Thurston showed that $\mathcal{ML}$ is homeomorphic to $\mathbb{R}^{2\convgenus}$, and $\mathcal{PML}$ is homeomorphic to $\mathbb{S}^{2\convgenus - 1}$.

The geometric intersection number  
\begin{equation}
\label{eq:intersection-number-WS}
i(t\alpha, s\beta) = ts\, i(\alpha, \beta)
\end{equation}
for weighted simple closed curves extends continuously to a bilinear function on $\mathcal{ML} \times \mathcal{ML}$.

A measured lamination $\lambda \in \mathcal{ML}$ is said to be \emph{minimal} if every leaf of $|\lambda|$ is dense in $|\lambda|$ (with respect to the topology induced from $\Sigma_{g,m}$).  
It is called \emph{filling} if every complementary region of $|\lambda|$ is either an ideal polygon or a once-punctured ideal polygon; equivalently, if $i(\lambda, \alpha) > 0$ for all $\alpha \in \mathcal{S}$.  
In this paper, a measured lamination $\lambda$ is said to be \emph{uniquely ergodic} if it is both minimal and filling, and if $\lambda' \in \mathcal{ML}$ satisfies $i(\lambda, \lambda') = 0$, then $\lambda' = t \lambda$ for some $t \ge 0$.

%A measured lamination $\lambda$ is said to be \emph{essentially complete} if each connected component of the complement of $|\lambda|$ is either an ideal triangle or a once-punctured ideal monogon when $(g,m) \ne (1,1)$, and a once-punctured bigon otherwise (cf.~\cite[Definition 9.5.1, Propositions 9.5.2 and 9.5.4]{Thuston-LectureNote}).  
%Essentially complete measured laminations form a generic subset of $\mathcal{ML}$ (cf.~\cite[Proposition 9.5.11]{Thuston-LectureNote}).

%\section{Extremal length geometry}
%\label{sec:ELG}
\subsection{Hubbard-Masur differentials and Extremal length}
\label{subsec:Hubbard-masur-theorem}
For $x = (M, f) \in \teich_{g,m}$ and $q \in \mathcal{Q}_{x}$,  
we define the \emph{vertical lamination} $v(q) \in \mathcal{ML}$  
of $q = q(z)dz^2$ by  
$$  
i(v(q), \alpha) = \inf_{\alpha' \in f(\alpha)} \int_{\alpha'} |{\rm Re}(\sqrt{q(z)}\,dz)|  
\quad (\alpha \in \mathcal{S}).  
$$  
We call $h(q) = v(-q)$ the \emph{horizontal lamination} of $q$.  
Hubbard and Masur \cite{MR523212} proved that the mapping  
\begin{equation}
\label{eq:Hubbard-Masur-homeo}
\mathcal{Q}_{x} \ni q \mapsto v(q) \in \mathcal{ML}
\end{equation}
is a homeomorphism for all $x \in \teich_{g,m}$.  
From \eqref{eq:Hubbard-Masur-homeo},  
for any $x \in \teich_{g,m}$ and $\lambda \in \mathcal{ML}$,  
there exists a unique $q_{\lambda,x} \in \mathcal{Q}_{x}$  
such that $v(q_{\lambda,x}) = \lambda$.  
We call $q_{\lambda,x}$ the \emph{Hubbard-Masur differential}  
for $\lambda$ on $x$.  
By definition, $q_{t\lambda,x} = t^2 q_{\lambda,x}$ for $t \ge 0$.

The \emph{extremal length} of $\lambda \in \mathcal{ML}$ on $x = (M, f) \in \teich_{g,m}$ is defined by  
$$  
\ext_{x}(\lambda) = \ext_{\lambda}(x) = \ext(x, \lambda) = \|q_{\lambda,x}\|,  
$$  
where we use the second symbol (resp. the third symbol) when we wish to emphasize that the extremal length is viewed as a function on $\teich_{g,m}$ (resp. a function of two variables on $\teich_{g,m} \times \mathcal{ML}$).

\subsection{Teichm\"uller rays and Teichm\"uller disks}
For $[\lambda] \in \mathcal{PML}$ and $x = (M, f) \in \teich_{g,m}$,  
the \emph{Teichm\"uller (geodesic) ray} $\ray_\lambda^x \colon [0, \infty) \to \teich_{g,m}$  
for $[\lambda]$ emanating from $x$ is defined as follows:  
for $t \ge 0$, let $h_t \colon M \to h_t(M)$ be the quasiconformal mapping with Beltrami differential  
$\tanh(t)\,|q_{\lambda,x}|/q_{\lambda,x}$.  
We set $\ray_\lambda^x(t) = (h_t(M), h_t \circ f)$.

Let $q \in \mathcal{Q}_x \setminus \{0\}$.  
For $|\lambda| < 1$, let $f_\lambda \colon M \to M_\lambda$ be the quasiconformal mapping with complex dilatation $\lambda\,\overline{q}/|q|$.  
Then,
$$
\Phi_q \colon (\mathbb{D}, d_{\mathbb{D}}) \to (M_\lambda, f_\lambda \circ f) \in (\teich_{g,m}, d_T)
$$
is a holomorphic and isometric mapping.  
We call $\Phi_q$ the \emph{Teichm\"uller disk} associated to the holomorphic quadratic differential $q = q(z)\,dz^2$.  
From the definition, we have
\begin{equation}
\label{eq:Teichmuller-rays-disks}
\ray_{v(e^{-i\theta}q)}^x(t) = \Phi_{e^{-i\theta}q}(\tanh(t)) = \Phi_q(\tanh(t)e^{i\theta})
\end{equation}
for $t \ge 0$.

\subsection{Thurston measure}
\label{subsec:Thurston-measure}
The \emph{Thurston measure} $\ThursM$ on $\mathcal{ML}$ is the unique, locally finite, mapping class group-invariant, ergodic measure supported on the set of filling measured laminations (cf. \cite{MR2424174}; see also \cite{MR1144770} and \cite{MR2415399}).  
%The Thurston measure satisfies the following scaling property: for any measurable set  
%$E \subset \mathcal{ML}$ and any $t > 0$,
%\begin{equation}
%\label{eq:multple-TH-measure}
%\ThursM(\{tF \mid F \in E\}) = t^{2\convgenus} \ThursM(E),
%\end{equation}
%where $tE = \{tF \mid F \in E\}$.

Let $x \in \teich_{g,m}$.  
Define  
$\mathcal{BML}_x = \{\lambda \in \mathcal{ML} \mid \ext_x(\lambda) \le 1\}$.  
We then define the \emph{unit sphere with respect to the extremal length function} as  
$$
\mathcal{SML}_x = \partial \mathcal{BML}_x = \{\lambda \in \mathcal{ML} \mid \ext_x(\lambda) = 1\}.
$$
The natural projection $\mathcal{ML} \setminus \{0\} \to \mathcal{PML}$  
induces a homeomorphism  
$$
\proje_x \colon \mathcal{SML}_x \to \mathcal{PML}.
$$

Let
\begin{align*}
\mathcal{SML}_x^{\mathrm{mf}} &= \{\eta \in \mathcal{SML}_x \mid \text{$\eta$ is minimal and filling}\}, \\
\mathcal{SML}_x^{\mathrm{ue}} &= \{\eta \in \mathcal{SML}_x^{\mathrm{mf}} \mid \text{$\eta$ is uniquely ergodic}\},
\end{align*}
and define $\mathcal{PML}^{\mathrm{mf}} = \proje_x(\mathcal{SML}_x^{\mathrm{mf}})$ and $\mathcal{PML}^{\mathrm{ue}} = \proje_x(\mathcal{SML}_x^{\mathrm{ue}})$.

We define a probability measure $\PThursM^x$ on $\mathcal{PML}$ via the cone construction:
\begin{equation}
\label{eq:Thurston-measure-SML}
\PThursM^x(E) =
\frac{
\ThursM(\{tG \mid G \in \proje_x^{-1}(E),\ 0 \le t \le 1\})}{\ThursM(\mathcal{BML}_x)}
\end{equation}
for any Borel set $E \subset \mathcal{PML}$.  

In this paper, we also refer to $\PThursM^x$ as the \emph{Thurston measure associated to} $x \in \teich_{g,m}$.  
It is known that
\begin{equation}
\label{eq:change-base-points}
\PThursM^x(E) =
\int_E \left(\frac{\ext(y, \lambda)}{\ext(x, \lambda)}\right)^{\convgenus}
d\PThursM^y([\lambda])
\end{equation}
for any measurable set $E \subset \mathcal{PML}$ and any $x, y \in \teich_{g,m}$  
(cf. \cite[\S2.3.1]{MR2913101}).

\section{The Bers slice}
\label{sec:Kleinian-surface-groups-and-the-Bers-slice}

\subsection{Bers slices}
\label{subsec:Bers-slice}
Fix $x_{0} = (M_{0}, f_{0}) \in \teich_{g,m}$,  
and let $\Gamma_{0}$ be the marked Fuchsian group acting on $\mathbb{H}$  
that uniformizes $M_{0}$ with the marking $\pi_1(\Sigma_{g,m}) \cong \Gamma_0$  
induced by $f_0$.
Let $A_{2}(\mathbb{H}^{*}, \Gamma_{0})$ denote the Banach space of automorphic forms  
on $\mathbb{H}^{*} = \hat{\mathbb{C}} \setminus \overline{\mathbb{H}}$  
of weight $-4$, equipped with the hyperbolic supremum norm  
$$
\|\varphi\|_\infty = \sup_{z \in \mathbb{H}^*} 4\, \mathrm{Im}(z)^2 |\varphi(z)|.
$$
For each $\varphi \in A_{2}(\mathbb{H}^{*}, \Gamma_{0})$,  
we can define a locally univalent meromorphic map $W_{\varphi}$ on $\mathbb{H}^{*}$  
and a monodromy homomorphism  
$\rho_{\varphi} \colon \Gamma_{0} \to \mathrm{PSL}_{2}(\mathbb{C})$  
such that the Schwarzian derivative of $W_{\varphi}$ is equal to $\varphi$,  
and $\rho_{\varphi}(\gamma) \circ W_{\varphi} = W_{\varphi} \circ \gamma$  
for all $\gamma \in \Gamma_{0}$.  
Let $\Gamma_{\varphi} = \rho_{\varphi}(\Gamma_{0})$.

The \emph{Bers slice} $\Bers{x_{0}}$ with basepoint $x_{0} \in \teich_{g,m}$  
is a domain in $A_{2}(\mathbb{H}^{*}, \Gamma_{0})$  
consisting of all $\varphi \in A_{2}(\mathbb{H}^{*}, \Gamma_{0})$  
such that $W_{\varphi}$ admits a quasiconformal extension to $\hat{\mathbb{C}}$.
The Bers slice $\Bers{x_{0}}$ is bounded and biholomorphically identified with $\teich_{g,m}$.  
Indeed, any $x \in \teich_{g,m}$ corresponds to some $\varphi$ for which $\Gamma_{\varphi}$  
is the marked quasifuchsian group uniformizing both $x_{0}$ and $x$  
(cf. \cite{MR0130972}).

The closure $\overline{\Bers{x_{0}}}$ of $\Bers{x_{0}}$ in $A_{2}(\mathbb{H}^{*}, \Gamma_{0})$  
is called the \emph{Bers compactification} of $\teich_{g,m}$.  
The boundary $\partial\Bers{x_{0}}$ is called the \emph{Bers boundary}.
For $\varphi \in \overline{\Bers{x_{0}}}$,  
the group $\Gamma_\varphi$ is a Kleinian surface group,  
with an isomorphism  
$\rho_\varphi \colon \pi_1(\Sigma_{g,m}) \cong \Gamma_0 \to \Gamma_\varphi$.

\subsection{Boundary groups without accidental parabolics}
\label{subsec:BoundarygroupswithoutAPT}
A boundary point $\varphi \in \partial \Bers{x_0}$ is called a \emph{cusp} if  
there exists a non-parabolic element $\gamma \in \Gamma_0$  
such that $\rho_{\varphi}(\gamma)$ is parabolic  
(cf. \cite[\S1]{MR0297992}).  
Such an element $\gamma$ or its image $\rho_{\varphi}(\gamma)$  
is called an \emph{accidental parabolic transformation}  
of $\varphi$ or $\Gamma_{\varphi}$  
(cf. \cite[\S4]{MR0297992}).
Let $\partial^{\mathrm{cusp}} \Bers{x_0}$ be the set of cusps in $\partial \Bers{x_0}$,  
and define  
$$
\partial^{\mathrm{mf}} \Bers{x_0} = \partial \Bers{x_0} \setminus \partial^{\mathrm{cusp}} \Bers{x_0}.
$$

Let $x_0 \in \teich_{g,m}$.  
Let $\mathcal{PML}^{\mathrm{mf}}$ denote the set of projective classes  
of minimal and filling measured laminations.  
By virtue of the Ending Lamination Theorem \cite{MR2925381}  
and the Double Limit Theorem of Thurston \cite{thurston1998hyperbolicstructures3manifoldsii},  
there exists a continuous, closed, and surjective map  
\begin{equation}
\label{eq:curve-complex-mininal-bers}
\Xi_{x_0} \colon \mathcal{PML}^{\mathrm{mf}} \to \partial^{\mathrm{mf}} \Bers{x_0}
\end{equation}
which assigns to each $[\lambda] \in \mathcal{PML}^{\mathrm{mf}}$  
the boundary group whose ending lamination is $|\lambda|$  
(see also \cite{MR2258749} and \cite{MR2582104}).  
For the definition of the ending lamination, see  
\cite{MR2630036}, \cite{MR2827011}, and \cite{Thuston-LectureNote}.

For simplicity, we write  
$$
\varphi_\lambda = \varphi_{\lambda, x_0} := \Xi_{x_0}([\lambda])
$$
for $[\lambda] \in \mathcal{PML}^{\mathrm{mf}}$.

The set $\mathcal{PML}^{\mathrm{mf}}$ contains  
a subset $\mathcal{PML}^{\mathrm{ue}}$ consisting of  
uniquely ergodic measured laminations.  
Let $\partial^{\mathrm{ue}} \Bers{x_0}$ denote the image of  
$\mathcal{PML}^{\mathrm{ue}}$ under the map $\Xi_{x_0}$.

The following proposition is well-known among experts.  
Indeed, it follows from the continuity of length functions  
and the Ending Lamination Theorem (cf. \cite{MR1791139}, \cite{MR2925381}, and \cite{MR1029395}).  
For a detailed proof, see \cite[Proposition 2.1]{MR4830064}.

\begin{proposition}
\label{prop:Teichmuller-limit}
Let $x_0 \in \teich_{g,m}$.  
For any $x \in \teich_{g,m}$ and $[\lambda] \in \mathcal{PML}^{\mathrm{mf}}$,  
the Teichmüller ray $\ray_\lambda^x$ converges to a totally degenerate  
group without accidental parabolics in $\partial^{\mathrm{mf}} \Bers{x_0}$,  
whose ending lamination is $|\lambda|$.
\end{proposition}

\subsection{Push-forward measures on the Bers boundaries}
\label{subsec:pushforwardmeasure-Bers-boundary}
For $x \in \teich_{g,m}$,  
we define a probability measure $\pushThursMBers_x$ on $\partial \Bers{x_0}$  
as the pushforward of the Thurston measure $\PThursM^x$ via the map $\Xi_{x_0} \colon \mathcal{PML}^{\mathrm{mf}} \to \partial \Bers{x_0}$:
\begin{equation}
\label{eq:defintion-pushforward-Th}
\int_{\partial \Bers{x_0}} f\, d\pushThursMBers_x
= \int_{\mathcal{PML}^{\mathrm{mf}}} f \circ \Xi_{x_0} \, d\PThursM^x
\end{equation}
for continuous functions $f$ on $\partial \Bers{x_0}$.
The superscript ``B" stands for the initial letter of the family name of Lipman Bers.

Masur \cite{MR644018} showed that $\mathcal{PML}^{\mathrm{ue}}$ has full measure in $\mathcal{PML}$  
with respect to $\PThursM^x$ for any $x \in \teich_{g,m}$.  
Hence, the composition $f \circ \Xi_{x_0}$ is defined almost everywhere on $\mathcal{PML}$.
Masur's result also implies that $\partial^{\mathrm{ue}} \Bers{x_0}$ is a set of full measure in $\partial \Bers{x_0}$  
with respect to the pushforward measure $\pushThursMBers_x$.

The author \cite{MR4633651} showed that the pushforward measure $\pushThursMBers_x$ coincides with the pluriharmonic measure on $\teich_{g,m}$  
with pole at $x \in \teich_{g,m}$ in the sense of Demailly,  
and that the complement $\partial \Bers{x_0} \setminus \partial^{\mathrm{mf}} \Bers{x_0}$ is null with respect to the pluriharmonic measure.

\section{Radial limits}
\label{sec:radial_limits}

\subsection{Radial limit theorem}
In \cite{MR4830064}, we observed the following:

\begin{proposition}[Radial limit theorem]
\label{prop:radial_limit_theorem}
For a bounded pluriharmonic function $u$ on $\teich_{g,m}$, there exists a full-measure set $\mathcal{E}_0 = \mathcal{E}_0(u) \subset \mathcal{PMF}$, depending only on $u$, with respect to the Thurston measure associated to some (and hence any) point $x \in \teich_{g,m}$, satisfying the following properties:
\begin{enumerate}
\item 
each element of $\mathcal{E}_0$ is minimal, filling, and uniquely ergodic;
\item
the radial limit $\lim_{t \to \infty} u(\ray_\lambda^x(t))$ exists for all 
$x \in \teich_{g,m}$ and $[\lambda] \in \mathcal{E}_0$; and
\item
the radial limit is independent of the choice of base point; namely,
$$
\lim_{t \to \infty} u(\ray_\lambda^{x_1}(t)) = \lim_{t \to \infty} u(\ray_\lambda^{x_2}(t))
$$
for all $[\lambda] \in \mathcal{E}_0$ and $x_1, x_2 \in \teich_{g,m}$.
\end{enumerate}
\end{proposition}

We define a bounded measurable function on $\partial \Bers{x_0}$ by
$$
u^*(\varphi_\lambda) = \begin{cases}
{\displaystyle \lim_{t \to \infty} u(\ray_\lambda^x(t))} & ([\lambda] \in \mathcal{E}_0), \\
0 & ([\lambda] \in \mathcal{PML} \setminus \mathcal{E}_0),
\end{cases}
$$
for a bounded pluriharmonic function $u$ on $\teich_{g,m}$, where
$\varphi_\lambda \in \partial \Bers{x_0}$ is the boundary group with the ending lamination $|\lambda|$.
We call $u^*$ the \emph{radial limit} of $u$ on $\partial \Bers{x_0}$.
By \eqref{prop:Teichmuller-limit}, $u^* \circ \Xi_{x_0} \in L^\infty(\mathcal{PML})$ is also regarded as the radial limit of $u$.

\subsection{Projectivization of $\mathcal{PML}$}
Let $x = (M,f) \in \teich_{g,m}$. We define an action $A_x$ of $\mathbb{S}^1$ on $\mathcal{PML}$ by
$$
A_x \colon \mathbb{S}^1 \times \mathcal{PML} \ni (e^{i\alpha}, [\lambda]) \mapsto [v(e^{-i\alpha} q_{\lambda,x})] \in \mathcal{PML}.
$$
Let $\mathbb{P}\mathcal{ML}_x$ denote the quotient space $\mathcal{PML} / \mathbb{S}^1 \cong \mathbb{CP}^{\convgenus - 1}$, and let $\Pi^x \colon \mathcal{PML} \to \mathbb{P}\mathcal{ML}_x$ be the natural projection.

The disintegration theorem~\cite{MR1484954} asserts that there exists a unique family $\{\boldsymbol{m}_t\}_{t \in \mathbb{P}\mathcal{ML}_x}$ of probability measures on $\mathcal{PML}$ such that:
\begin{itemize}
\item
each $\boldsymbol{m}_t$ is supported on the fiber $(\Pi^x)^{-1}(t)$;
\item
for any non-negative measurable function $f$ on $\mathcal{PML}$, the function
$$
\mathbb{P}\mathcal{ML}_x \ni t \mapsto \int_{\mathcal{PML}} f([\lambda]) \, d\boldsymbol{m}_t([\lambda])
$$
is measurable;
\item
for any non-negative measurable function $f$ on $\mathcal{PML}$,
$$
\int_{\mathbb{P}\mathcal{ML}_x}
\int_{\mathcal{PML}} f([\lambda]) \, d\boldsymbol{m}_t([\lambda]) \, d\boldsymbol{\nu}^x(t)
= \int_{\mathcal{PML}} f([\lambda]) \, d\PThursM^x([\lambda]),
$$
where $\boldsymbol{\nu}^x = (\Pi^x)_*(\PThursM^x)$ denotes the pushforward measure.
\end{itemize}

Let $t \in \mathbb{P}\mathcal{ML}_x$ and choose $q_t \in \mathcal{q}_t$ such that $\Pi^x([v(q_t)]) = t$. Define the embedding
$$
\boldsymbol{\phi}_t \colon \mathbb{S}^1 \ni e^{i\theta} \mapsto [v(e^{-i\theta} q_t)] \in (\Pi^x)^{-1}(t).
$$
This embedding is equivariant with respect to the $\mathbb{S}^1$-actions on both $\mathbb{S}^1$ and the fiber $(\Pi^x)^{-1}(t)$; that is,
$$
\boldsymbol{\phi}_t(e^{i\alpha} e^{i\theta}) = [v(e^{-i(\alpha+\theta)} q_t)] 
= [v(e^{-i\alpha}(e^{-i\theta} q_t))] = A_x(e^{i\alpha}, \boldsymbol{\phi}_t(e^{i\theta})).
$$
Hence, the fiber $(\Pi^x)^{-1}(t)$ inherits the angle measure $\Theta_t$, defined as the pushforward of the normalized Lebesgue measure $d\theta / 2\pi$ on $\mathbb{S}^1$ via $\boldsymbol{\phi}_t$.

\begin{proposition}[{\cite[Proposition 3.2]{MR4830064}}]
\label{prop:dis_inte_angle}
$\boldsymbol{m}_t = \Theta_t$ for almost every $t \in \mathbb{P}\mathcal{ML}_x$.
\end{proposition}

In the proof of \Cref{prop:dis_inte_angle}, the invariance of the Thurston measure under the action $A_x$, established by Dumas~\cite{MR3413977}, plays a crucial role.

\section{Proofs}
\label{sec:Proofs_theorems}
\subsection{Proof of \Cref{thm:main1}}
Let $u$ be a bounded pluriharmonic function on $\teich_{g,m}$, and let $u^*$ denote its radial limit.
Define $U^* = u^* \circ \Xi_{x}$, which is a bounded measurable function on $\mathcal{PML}$. Then we have
\begin{align*}
\int_{\partial \Bers{x_0}} u^*(\varphi) \, d\pushThursMBers_x(\varphi)
&= \int_{\mathcal{PML}} U^*([\lambda]) \, d\PThursM^x([\lambda]) \\
&= \int_{\mathbb{P}\mathcal{ML}_x} \int_{\mathcal{PML}} U^*([\lambda]) \, d\Theta_t([\lambda]) \, d\boldsymbol{\nu}^x(t),
\end{align*}
by the disintegration theorem and \Cref{prop:dis_inte_angle}.

For $t \in \mathbb{P}\mathcal{ML}_x$, let $q_t \in \mathcal{Q}_x$ be the corresponding unit-norm quadratic differential.  
By \eqref{eq:Teichmuller-rays-disks}, \Cref{prop:Teichmuller-limit}, and \Cref{prop:dis_inte_angle}, we have
\[
U^*([v(e^{-i\theta} q_t)]) = u^*\!\left(\Xi_{x}([v(e^{-i\theta} q_t)])\right) = u^*\!\left(\varphi_{v(e^{-i\theta} q_t)}\right),
\]
which coincides with the radial limit
\[
(u \circ \Phi_{q_t})^*(e^{i\theta}) = \lim_{r \to 1} u \circ \Phi_{q_t}(r e^{i\theta})
\]
of the harmonic function $u \circ \Phi_{q_t}$ on $\mathbb{D}$, for almost every $e^{i\theta} \in \mathbb{S}^1$ and for almost every $t \in \mathbb{P}\mathcal{ML}_x$.
Since each $u \circ \Phi_{q_t}$ is bounded harmonic, by \cite[p.~38, Corollary]{MR1102893}, we have
\begin{align*}
\int_{\mathcal{PML}} U^*([\lambda]) \, d\Theta_t([\lambda])
&= \int_0^{2\pi} (u \circ \Phi_{q_t})^*(e^{i\theta}) \frac{d\theta}{2\pi} = u \circ \Phi_{q_t}(0) = u(x)
\end{align*}
for almost every $t \in \mathbb{P}\mathcal{ML}_x$.
Therefore, we obtain
\begin{align*}
\int_{\partial \Bers{x_0}} u^*(\varphi) \, d\pushThursMBers_x(\varphi)
&= \int_{\mathbb{P}\mathcal{ML}_x} \int_{\mathcal{PML}} U^*([\lambda]) \, d\Theta_t([\lambda]) \, d\boldsymbol{\nu}^x(t) \\
&= \int_{\mathbb{P}\mathcal{ML}_x} u(x) \, d\boldsymbol{\nu}^x(t) = u(x).
\end{align*}
\subsection{Proof of \Cref{coro:main1dash}}
We only check that the Poisson integral $\mathbb{P}(g)$ of an integrable function $g$ on $\mathcal{PML}$ is of class $C^1$.

First, notice that when $g$ is a bounded measurable function, for any $x \in \teich_{g,m}$, we have
\begin{align*}
|\mathbb{P}(g)(x)| 
&\le \int_{\partial \Bers{x_0}} |g(\varphi)| \, d\pushThursMBers_x(\varphi)
\le \|g\|_\infty \int_{\partial \Bers{x_0}} d\pushThursMBers_x(\varphi) = \|g\|_\infty.
\end{align*}
Hence, $\mathbb{P}(g)$ is bounded.

Now suppose $g$ is integrable on $\partial \Bers{x_0}$ with respect to $\pushThursMBers_x$. From \eqref{eq:change-base-points} and \eqref{eq:defintion-pushforward-Th}, the Poisson integral \eqref{eq:PI} of $g$ satisfies
\begin{align*}
\mathbb{P}(g)(x)
&= \int_{\partial \Bers{x_0}} g(\varphi) \, d\pushThursMBers_x(\varphi)
= \int_{\mathcal{PML}} g \circ \Xi_{x_0}([\lambda]) \, d\PThursM^x([\lambda]) \\
&= \int_{\mathcal{PML}} g \circ \Xi_{x_0}([\lambda])
\left(
\frac{\ext(x_0, \lambda)}{\ext(x, \lambda)}
\right)^{\convgenus}
d\PThursM^{x_0}([\lambda])
\end{align*}
for $x \in \teich_{g,m}$. Since the extremal length function is of class $C^1$, so is $\mathbb{P}(g)$.

By Gardiner's formula, we have
\begin{align*}
\left.\partial \mathbb{P}(g)\right|_{x}
&= \convgenus
\int_{\mathcal{PML}} g \circ \Xi_{x_0}([\lambda])
\left(
\frac{\ext(x_0, \lambda)}{\ext(x, \lambda)}
\right)^{\convgenus}
\frac{q_{\lambda, x}}{\ext(x, \lambda)}
d\PThursM^{x_0}([\lambda]), \\
\left.\overline{\partial} \mathbb{P}(g)\right|_{x}
&= \convgenus
\int_{\mathcal{PML}} g \circ \Xi_{x_0}([\lambda])
\left(
\frac{\ext(x_0, \lambda)}{\ext(x, \lambda)}
\right)^{\convgenus}
\frac{\overline{q_{\lambda, x}}}{\ext(x, \lambda)}
d\PThursM^{x_0}([\lambda])
\end{align*}
as $(1,0)$- and $(0,1)$-forms, respectively.

\begin{remark}
For a formulation of the complex cotangent spaces (i.e., the $(1,0)$- and $(0,1)$-cotangent spaces) over Teichmüller space in terms of holomorphic and anti-holomorphic quadratic differentials, see \cite{bounded-Function_Ergodic}, for instance.
\end{remark}

\subsection{Value distribution of psh functions}
Applying the disintegration argument given above, we discuss the F. and M. Riesz-type inequality (\Cref{thm:main2}) for non-negative and bounded plurisubharmonic functions on $\teich_{g,m}$.

\subsubsection{A lemma}

First, we provide a proof of the following result, which is well-known among experts.

\begin{lemma}
\label{lem:subharmonic_maximum}
Let $v$ be a non-negative bounded subharmonic function on the unit disk $\mathbb{D}$ such that $\log v$ is also subharmonic on $\mathbb{D}$.
Suppose there exists a measurable decomposition $\{E_k\}_{k=1}^\infty$ of $\partial \mathbb{D}$ and positive numbers $m_k$ ($k \in \mathbb{N}$) such that 
\[
\limsup_{r \to 1} v(re^{i\theta}) \le m_k \quad \text{for all } e^{i\theta} \in E_k.
\]
Then,
\[
v(0) \le \prod_{k=1}^\infty m_k^{\Theta(E_k)},
\]
where $\Theta(E_k) = \int_{E_k} \frac{d\theta}{2\pi}$ is the normalized Lebesgue measure on $\partial \mathbb{D}$.
\end{lemma}

\begin{proof}
Since $\log v$ is subharmonic and bounded, by Fatou's theorem,
\begin{align*}
\log v(0)
&\le \limsup_{r \to 1} \int_0^{2\pi} \log v(re^{i\theta}) \frac{d\theta}{2\pi} \\
&\le \int_0^{2\pi} \limsup_{r \to 1} \log v(re^{i\theta}) \frac{d\theta}{2\pi}
\le \sum_{k=1}^\infty \Theta(E_k) \log m_k,
\end{align*}
where the last term is bounded.
\end{proof}

\subsubsection{Proof of \Cref{thm:main2}}
For $t \in \mathbb{P}\mathcal{ML}_x$ and $k \in \mathbb{N}$, let $E_k' = \Xi_x^{-1}(E_k)$ and
\[
E_k'(t) = E_k' \cap (\Pi^x)^{-1}(t).
\]
Then $\{E_k'\}_{k=1}^\infty$ is a measurable decomposition of $\mathcal{PML}$.

Let $q_t \in \mathcal{Q}_x \setminus \{0\}$ be the quadratic differential corresponding to $t \in \mathbb{P}\mathcal{ML}_x$.
Fix $s \in \mathbb{N}$ and $M > 0$ such that $v(z) \le M$ and $m_k \le M$ for all $z \in \teich_{g,m}$ and $k \in \mathbb{N}$. Applying \Cref{lem:subharmonic_maximum} to $v \circ \Phi_{q_t}$, we obtain
\[
\log v(x) \le \sum_{k=1}^s \Theta(E_k'(t)) \log m_k + \left(1 - \sum_{k=1}^s \Theta(E_k'(t)) \right) \log M.
\]
By Fubini's theorem,
\begin{align*}
\log v(x) &\le \sum_{k=1}^s \log m_k \int_{\mathbb{P}\mathcal{ML}_x} \Theta(E_k'(t)) \, d\boldsymbol{\nu}^x(t) \\
&\quad + \log M \int_{\mathbb{P}\mathcal{ML}_x} \left(1 - \sum_{k=1}^s \Theta(E_k'(t)) \right) d\boldsymbol{\nu}^x(t) \\
&= \sum_{k=1}^s \PThursM^x(E_k') \log m_k + \left(1 - \sum_{k=1}^s \PThursM^x(E_k') \right) \log M \\
&= \sum_{k=1}^s \PThursM^x(E_k') \log m_k + \PThursM^x\left(\mathcal{PML} \setminus \bigcup_{k=1}^s E_k' \right) \log M.
\end{align*}
Since the right-hand side is non-increasing in $s$, taking the limit as $s \to \infty$ yields
\[
\log v(x) \le \sum_{k=1}^\infty \PThursM^x(E_k') \log m_k = \sum_{k=1}^\infty \pushThursMBers_x(E_k) \log m_k,
\]
which gives the desired result.

\section{Mean value theorem}
\label{sec:MVT-JM}
Let $\Omega$ be a bounded domain in a complex Euclidean space.
A positive regular Borel measure $\mu$ on $\overline{\Omega}$ is said to be a \emph{Jensen measure} with barycenter $z \in \overline{\Omega}$ for continuous plurisubharmonic functions if
\[
u(z) \le \int_{\overline{\Omega}} u \, d\mu
\]
holds for every continuous function $u$ on $\overline{\Omega}$ that is plurisubharmonic on $\Omega$ (cf.\ e.g., \cite{MR2207205, MR1821089}).

Let $x \in \teich_{g,m}$ and $r > 0$.
Define
\[
B(x,r) = \{ y \in \teich_{g,m} \mid d_T(x,y) \le r \}, \quad S(x,r) = \partial B(x,r).
\]
Let $\boldsymbol{\mu}_{x;r}$ be the pushforward of the measure $\PThursM^x$ on $\mathcal{PML}$ to $S(x,r)$ via the homeomorphism
\[
\mathcal{PML} \ni [\lambda] \mapsto \ray_\lambda^x(r) \in S(x,r).
\]

\begin{theorem}[Jensen measures and the mean value property]
\label{thm:Jencen_mean_value}
The measure $\boldsymbol{\mu}_{x;r}$ is a Jensen measure on $B(x,r)$ with barycenter $x$ in the sense that
\[
u(x) \le \int_{B(x,r)} u \, d\boldsymbol{\mu}_{x;r} = \int_{S(x,r)} u \, d\boldsymbol{\mu}_{x;r}
\]
for every continuous function $u$ on $B(x,r)$ that is plurisubharmonic in the interior.
Moreover, if $u$ is pluriharmonic in the interior, then the mean value property
\begin{equation}
\label{eq:pluriharmonic-mean-value}
u(x) = \int_{S(x,r)} u \, d\boldsymbol{\mu}_{x;r}
\end{equation}
holds.
\end{theorem}

\begin{proof}
For each $t \in \mathbb{P}\mathcal{ML}_x$, let $q_t \in \mathcal{Q}_x \setminus \{0\}$ be the corresponding holomorphic quadratic differential.
Since $u \circ \Phi_{q_t}$ is continuous on the closed disk $\{|\lambda| \le r\}$ and subharmonic in its interior, we have
\begin{equation}
\label{eq:subharmonic}
u(x) = u \circ \Phi_{q_t}(0) \le \int_0^{2\pi} u \circ \Phi_{q_t}(r e^{i\theta}) \frac{d\theta}{2\pi}.
\end{equation}
By the disintegration theorem and \Cref{prop:dis_inte_angle}, 
\begin{align}
\int_{B(x,r)} u \, d\boldsymbol{\mu}_{x;r}
&= \int_{S(x,r)} u \, d\boldsymbol{\mu}_{x;r} = \int_{\mathcal{PML}} u(\ray_\lambda^x(r)) \, d\PThursM^x([\lambda])
\label{eq:MVP}
\\
&= \int_{\mathbb{P}\mathcal{ML}_x} \left( \int_{\mathcal{PML}} u(\ray_\lambda^x(r)) \, d\Theta_t([\lambda]) \right) d\boldsymbol{\nu}^x(t)
\nonumber \\
&= \int_{\mathbb{P}\mathcal{ML}_x} \left( \int_0^{2\pi} u \circ \Phi_{q_t}(r e^{i\theta}) \frac{d\theta}{2\pi} \right) d\boldsymbol{\nu}^x(t)
\nonumber \\
&\ge \int_{\mathbb{P}\mathcal{ML}_x} u(x) \, d\boldsymbol{\nu}^x(t) = u(x).
\nonumber
\end{align}
If $u$ is pluriharmonic in the interior of $B(x,r)$, then $u \circ \Phi_{q_t}$ is harmonic on $\{ |z| < r \}$.
By the classical mean value property for harmonic functions, equality holds in \eqref{eq:subharmonic}.
Hence the equality in \eqref{eq:pluriharmonic-mean-value} follows from the above computation.
\end{proof}

\section{Conclusion}
\label{sec:conclusion}

To conclude this paper, we present the following conjectures and questions related to our results.

\subsection{}
Developing function theory on Teichm\"uller space for applications to low-dimensional topology remains an important open problem. In a forthcoming paper \cite{bounded-Function_Ergodic}, we apply \Cref{thm:main1} and \Cref{coro:main1dash} to study the dynamics of subgroups of the mapping class group.

Fatou~\cite{MR1555035} showed that any bounded harmonic function on $\mathbb{D}$ admits a non-tangential limit (cf.\ e.g., \cite{MR0114894}). It is natural to ask whether the radial limit given in \Cref{prop:radial_limit_theorem} can be strengthened to a non-tangential limit. An affirmative answer to this question would be significant for understanding the conical limit sets of subgroups of the mapping class group, particularly the limit sets of convex cocompact subgroups (cf.\ \cite{MR1914566}, \cite{MR2465691}, and \cite{MR624833}).

\subsection{}
In his proof of the Poisson integral formula on a hyperconvex domain $\Omega \subset \mathbb{C}^n$, Demailly~\cite{MR881709} applied the Lelong-Jensen formula:
\[
\int_{S(r)} V\, d\mu_{\varphi,r} - \int_{B(r)} V\, (dd^c \varphi)^n = \int_{B(r)}(r-\varphi) dd^c V \wedge (dd^c \varphi)^{n-1},
\]
to define a family $\{\mu_{\varphi,r}\}_{r<0}$ of Monge-Amp\`ere measures associated to $\varphi$.  
Here, $V$ is a plurisubharmonic function on $\Omega$, $\varphi$ is a plurisubharmonic exhaustion function, $S(r) = \{\varphi = r\}$, and $B(r) = \{\varphi < r\}$.  
Demailly considered the family of Monge-Amp\`ere measures associated to the pluricomplex Green function and took the limit as $r \to 0$ to construct the pluriharmonic measure.

It is a natural conjecture that the family $\{\boldsymbol{\mu}_{x;r}\}_{r>0}$ in \S\ref{sec:MVT-JM} arises from the Monge-Amp\`ere measures associated to the pluricomplex Green function on $\teich_{g,m}$.  
It is also reasonable to ask whether the measure $\boldsymbol{\mu}_{x;r}$ coincides with the pluriharmonic measure on the open ball $\{ y \in \teich_{g,m} \mid d_T(x,y) < r \}$ with pole at $x$, in the sense of Demailly.

\subsection{}
Masur~\cite[Proposition 2.1]{MR1383491} observed a mean value property for pluriharmonic functions on $\teich_g$ via a random walk, which raises the question of how the pluriharmonic measure described here relates to the hitting measure of Masur's random walk. Similarly, Kakutani~\cite{MR0014647} showed that the harmonic measure is the hitting measure of Brownian motion on $\mathbb{D}$. This motivates the question of whether there exists a stochastic process -- specifically, a Brownian motion or more generally a Markov process -- on the Bers slice or Thurston compactification whose hitting measure coincides with the pluriharmonic measure considered in this paper.

\bibliographystyle{plain}
\bibliography{Miyachi_bdd_ph_Poisson.bbl}

\end{document}